\documentclass[reqno,a4paper]{amsart}
\usepackage{enumerate,amsmath,amssymb}
\usepackage{pstricks,pst-node,pst-coil,pst-plot}
\usepackage{hyperref}
\newtheorem*{question}{Question}

\theoremstyle{plain}

\newtheorem{thm}{Theorem}[section]
\newtheorem{theorem}[thm]{Theorem}
\newtheorem{cor}[thm]{Corollary}
\newtheorem{corollary}[thm]{Corollary}

\newtheorem{lemma}[thm]{Lemma}

\newtheorem{proposition}[thm]{Proposition}

\theoremstyle{remark}

\theoremstyle{definition}

\def\al{{\alpha}}

\def\om{{\omega}}

\def\la{{\lambda}}
\let\La\Lambda

\def\si{{\sigma}}
\def\Si{{\Sigma}}
\def\ga{{\gamma}}
\def\epsilon{{\varepsilon}}
\def\ep{{\varepsilon}}

\def\phi{{\varphi}}

\DeclareMathAlphabet{\doba}{U}{msb}{m}{n}

\gdef\mH{\doba{H}}

\gdef\mR{\doba{R}}
\gdef\mS{\doba{S}}

\def\cE{\mathcal{E}}
\def\cF{\mathcal{F}}

\def\vol{{\mathop{\rm vol}}}
\def\Scal{{s}}
\let\scal\Scal
\let\pa\partial
\let\na\nabla
\let\<\langle
\let\>\rangle
\let\ti\tilde
\let\witi\widetilde

\newcommand{\definedas}{\mathrel{\raise.095ex\hbox{\rm :}\mkern-5.2mu=}}

\begin{document}


\title
[The conformal Yamabe constant \break of product manifolds]
{The conformal Yamabe constant \break of product manifolds}

\author{Bernd Ammann}
\address{Fakult\"at f\"ur  Mathematik \\
Universit\"at Regensburg \\
93040 Regensburg \\
Germany}
\email{bernd.ammann@mathematik.uni-regensburg.de}

\author{Mattias Dahl}
\address{Institutionen f\"or Matematik \\
Kungliga Tekniska H\"ogskolan \\
100 44 Stockholm \\
Sweden}
\email{dahl@math.kth.se}

\author{Emmanuel Humbert}
\address{Laboratoire de Math\'ematiques et Physique Th\'eorique \\
Universit\'e de Tours \\
Parc de Grandmont \\
37200 Tours - France \\}
\email{Emmanuel.Humbert@lmpt.univ-tours.fr}


\begin{abstract}
Let $(V,g)$ and $(W,h)$ be compact Riemannian manifolds of dimension
at least $3$. We derive a lower bound for the conformal Yamabe
constant of the product manifold $( V \times W, g+h)$ in terms of the
conformal Yamabe constants of $(V,g)$ and $(W,h)$.
\end{abstract}

%

\date{\today}


\thanks{B. Ammann was partially supported by DFG Sachbeihilfe AM 144/2-1,}
\thanks{M. Dahl was partially supported by the Swedish Research Council,}
\thanks{E. Humbert was partially supported by ANR-10-BLAN 0105.}

\maketitle

\setcounter{tocdepth}{1}

\section{Introduction}

\subsection{The Yamabe functional, constant scalar curvature metrics,
  and Yamabe metrics}

For a Riemannian manifold $(M,G)$ we denote the scalar curvature by
$\Scal^G$, Laplace operator by $\Delta^G$, and volume form by $dv^G$. In general
the dependence on the Riemannian metric is denoted by the metric as a
superscript.

For integers $m \geq 3$ we set $a_m \definedas \frac{4(m-1)}{m-2}$
and $p_m \definedas \frac{2m}{m-2}$. Let $C^\infty_c(M)$ denote the
space of compactly supported smooth functions on $M$. For a Riemannian
manifold $(M,G)$ of dimension $m \geq 3$ we define the Yamabe
functional by
\begin{equation*}
\cF^G (u)
\definedas
\frac{
\int_M \left( a_m |du|_G^2 + \Scal^G u^2 \right) \, dv^G}
{\left( \int_M |u|^{p_m} \, dv^G \right)^{\frac{2}{p_m}}} ,
\end{equation*}
where $u \in C^\infty_c(M)$ does not vanish identically. The
{\it conformal Yamabe constant} $\mu(M,G)$ of $(M,G)$ is defined by
\begin{equation*}
\mu(M,G) \definedas
\inf_{u \in C_c^{\infty}(M), u \not\equiv 0} \cF^G(u).
\end{equation*}
The conformal Yamabe constant is usually defined only for compact
manifolds; here we also allow non-compact manifolds in the definition.
This will turn out to be essential for studying surgery formulas for
Yamabe invariants of compact manifolds; see
Subsection~\ref{subsec.surg}. Also notice that the conformal Yamabe
constant for non-compact manifolds has been studied for instance in
\cite{kim:00} and \cite{grosse:p09}.

For compact $M$ one easily sees that
$\lim_{\ep\to 0} \cF^G (\sqrt{u^2+\ep^2}) = \cF^G (u)$; thus we obtain
\begin{equation*}
\mu(M,G)
=
\inf_{u\in C^\infty_+(M)} \cF^G (u) > -\infty,
\end{equation*}
where $C^\infty_+(M)$ denotes the space of positive smooth
functions. According to the resolution of the Yamabe problem
\cite{trudinger:68, aubin:76, schoen:84}, see for example
\cite{lee.parker:87} for a good overview article,
this infimum is always attained by a positive smooth function if
$M$ is a compact manifold.

For a compact manifold $M$ one also defines for any metric $G$
the (normalized) Einstein-Hilbert functional $\cE$ as
\begin{equation*}
\cE(G)
\definedas
\frac{\int_M \scal^G\, dv^G}{\vol^G(M)^{\frac{m-2}m}}.
\end{equation*}
These functionals are closely related to each other; namely, if $u>0$
and $\witi G = u^{4/(m-2)} G$, then
\begin{equation*}
\cE(\witi G) = \cF^G (u).
\end{equation*}
{}From the discussion above it follows that the functional $\cE$ always
attains its infimum in each conformal class $[G]$. Such
minimizing metrics are called \emph{Yamabe metrics}. Obviously $\witi
G$ is a Yamabe metric if and only if $\la \witi G$ is a Yamabe metric
for any $\la>0$. Thus any conformal class on a compact manifold
carries a Yamabe metric of volume~$1$. Yamabe metrics $\witi G$ are
stationary points of $\cE$, restricted to the conformal class, and
thus satisfy an Euler-Lagrange equation. This Euler-Lagrange equation
says precisely that the scalar curvature of $\witi G$ is constant. One
also sees that $\mu(M,G)$ is positive if and only if $[G]$ contains a
metric of positive scalar curvature.

We denote the standard flat metric on $\mR^m$ by $\xi^m$. The induced
metric on the sphere $S^m \subset \mR^{m+1}$ will be denoted by
$\rho^m$. This metric is a Yamabe metric, and the whole orbit of
$\rho^m$ under the action of the M\"obius group ${\rm Conf}(S^m) =
{\rm  PSO}(m+1,1)$ consists of Yamabe metrics. Thus
$\mS^m \definedas (S^m, \rho^m)$ carries a non-compact space of Yamabe
metrics of volume~$1$. It follows easily that $\mu(S^m,\rho^m) =
m(m-1) \om_m^{2/m}$, where $\om_m$ is the volume of $(S^m,\rho^m)$.

To determine the conformal Yamabe constant $\mu(M,G)$ it is obviously
sufficient to know a Yamabe metric in $[G]$ explicitly.
This is often easy to obtain in the following cases:
\begin{enumerate}[(i)]
\item \label{item1}
If $M$ is compact and connected, and if $\mu(M,G)\leq 0$, then it
follows from the maximum principle that there is a unique metric
$g_0 \in [G]$ of volume~$1$ and constant scalar curvature $s_0$. This
then implies that $\mu(M,G)=s_0$.
\item \label{item2}
Assume that $(M,G)$ is a connected compact Einstein manifold, and
$(M,G)$ is non-isometric to $(S^m,\la \rho^m)$ for any $\la >0$. Then
it was proven by Obata \cite[Prop.~6.2]{obata:71.72} that $G$ contains
a unique metric $g_0$ of volume~$1$ and constant scalar curvature
$s_0$. Again  $\mu(M,G)=s_0$.
\item \label{item3}
If $(M,\ti G)$ is a metric of constant scalar curvature $\ti s$ which
is close in the $C^{2,\al}$-topology to an Einstein manifold $(M,G)$
as in \eqref{item2}, then it is proven by \cite[Theorem C]{boehm.wang.ziller:04}
that $\ti G$ is also a Yamabe metric, and thus $\mu(M,\ti G) =
\ti s \; \vol(M,\ti G)^{2/m}$. This applies for example to
$(M,\ti G) = (S^m,\rho^m)\times (S^m, (1+\ep)\rho^m)$, $\ep$ close to
$0$. It also follows from the arguments in \cite{boehm.wang.ziller:04}
that $\ti G$ is then (up to rescaling) the only Yamabe metric in
$[\ti G]$. However, it is hard to decide in this situation whether
$\ti G$ is (up to rescaling) the only metric of constant scalar
curvature in $[\ti G]$. An affirmative answer to this problem was
given recently in \cite[Theorem~5]{lima.piccione.zedda:p11} if at
least one of the following additional conditions is satisfied:
(a)~$m\leq 7$; (b)~$m\leq 24 $ and $M$ is spin; or (c)~$|W|+|\na W|$
is a positive function.
\end{enumerate}




We have seen that in some particular cases, $\mu(M,g)$ can be
explicitly calculated. In general the determination of
$\mu(M,g)$ is a difficult task, as in most cases it is unclear
whether a given constant scalar curvature metric in~$[G]$ is a Yamabe
metric. The functionals $\cE|_{[G]}$ and $\cF^G |_{C^\infty_+(M)}$ may
have non-mini\-mi\-zing stationary points. These stationary points are
thus metrics of constant scalar curvature which are not Yamabe
metrics. The simplest such example, extensively discussed by Schoen
\cite{schoen:91} for $w=1$, is the metric $G = \rho^v + \la \rho^w$ on
$S^v \times S^w$, $v\geq 2$, which has constant scalar curvature
$v(v-1) + (1/\la) w(w-1)$, but which is not a Yamabe metric for
sufficiently large $\la$. This is due to the fact that
$\mu(M,G) \leq \mu(\mS^m)$, which follows from a standard
test function argument, whereas $\cE( \rho^v + \la \rho^w )
\to \infty$  as $\la \to\infty$ when $v\geq 2$.

In conclusion, if $(M,G)$ is an explicitly given compact manifold of
constant scalar curvature, then the calculation of $\mu(M,G)$ is easy if
either $(M,G)$ is Einstein or if $\mu(M,G)\leq 0$, but in general it
can be a hard problem.

\subsection{Product manifolds}

We now consider Riemannian product manifolds; that is, for Riemannian
manifolds $(V,g)$ and $(W,h)$ of dimensions $v$ and $w$,
we equip $M = V \times W$ with the product metric $G = g + h$, or more
generally $G = g + \la h$, where $\la > 0$. We ask the following
question.

\begin{question}
Suppose $V$ and $W$ are compact and equipped with Yamabe metrics $g$
and $h$. Let $\la > 0$. Is then $g + \la h$ also a Yamabe metric?
\end{question}

{}From the discussion on uniqueness above it follows that the answer
is yes,
\begin{itemize}
\item if $v,w\geq 3$, $\mu(V,g)\leq 0$ and $\mu(W,h)\leq 0$;
\item or if $v,w\geq 3$, $\mu(V,g)> 0$ and $\mu(W,h)< 0$ for $\la>0$
small enough;
\item or if $(V,g)$ and $(W,h)$ are both Einstein with
$\frac{1}{v} \Scal^g$ close to $\frac{1}{\la w} \Scal^h$.
\end{itemize}

If the answer to the above question 
is yes, then one deduces
\begin{equation}\label{naive.product.formula}
\mu(V\times W,g+\la h)
=
\left(\frac{\mu(V,g)}{\vol^g (V)^{2/v}}
+ \frac{\mu(W,h)}{\vol^{\la h} (W)^{2/w}}\right)
\Bigl( \vol^g (V) \vol^{\la h} (W) \Bigr)^{\frac 2{v+w}}.
\end{equation}
On the other hand if $g$ has positive scalar curvature, then
$\cE(g+\la h)\to \infty$ for $\la\to \infty$; thus $g + \la h$ is not
a Yamabe metric for large $\la$. This applies, in particular, to the
cases $\mu(V,g)>0$, $v\geq 3$, or if $(V,g) = (S^2, \rho^2)$.

\subsection{An intuitive---but incorrect---argument in the positive case}

Now we assume $v,w\geq 3$, $\mu(V,g)>0$, and $\mu(W,h)>0$. We already
explained why $g+\la h$ is not a Yamabe metric for large (and small)
$\la>0$, and as a consequence Equation~\eqref{naive.product.formula} cannot
be true for all $\la>0$. Despite this fact, assume for a moment
that \eqref{naive.product.formula} were true for all $\la>0$. We then
could minimize over $\la$, and we would obtain
\begin{equation} \label{naive.product.concl}
\inf_{\la\in (0,\infty)} \mu(V\times W,g+\la h)
=
(v+w) \left(\frac{\mu(V,g)}v\right)^{\frac v{v+w}}
\left(\frac{\mu(W,h)}w\right)^{\frac w{v+w}}.
\end{equation}

\subsection{Main result}

Although the \emph{naive derivation} of formula~\eqref{naive.product.concl}
used incorrect assumptions, our main result, Theorem~\ref{main},
will tell us that the \emph{formula itself} is correct up to a factor
\begin{equation*}
\ep_{v,w} =
\frac{a_{v+w}}{{a_v}^{v/(v+w)}{a_w}^{w/(v+w)}}
< 1,
\end{equation*}
assuming the mild condition \eqref{assump}.

More precisely, we assume that $V$ and $W$ are compact manifolds
of dimension at least $3$, with Yamabe metrics $g$ and $h$ of positive
conformal Yamabe constant. In particular, condition~\eqref{assump} is
satisfied.
Then Theorem~\ref{main} implies that
\begin{equation*}
\ep_{v,w} \leq
\frac{\inf_{\la\in (0,\infty)} \mu(V\times W,g+\la h)}
{(v+w) \left(\frac{\mu(V,g)}v\right)^{\frac v{v+w}}
\left(\frac{\mu(W,h)}w\right)^{\frac w{v+w}}}
\leq 1 .
\end{equation*}
Note that $\ep_{v,w}\to 1$ for $v,w\to \infty$.
See Figure~\ref{fig.ep} for some values of $\ep_{v,w}$.

\begin{figure}
\begin{tabular}{r||l|l|l|l|l}
  $\ep_{v,w}$   & $w=3$       & $w=4$      & $w=5$ & $w=6$ & $w=7$ \\
\hline\hline
$v = 3$ & 0.625     & 0.7072.. & 0.7515.. & 0.7817.. & 0.8042.. \\
   4 & 0.7072..  & 0.7777.. & 0.8007.. & 0.8367.. & 0.8537.. \\
   5 & 0.7515..  & 0.8007.. & 0.8427.. & 0.8631.. & 0.8772.. \\
   6 & 0.7817..  & 0.8367.. & 0.8631.. & 0.88     & 0.8921.. \\
   7 & 0.8042..  & 0.8537.. & 0.8772.. & 0.8921.. & 0.9027.. \\
\end{tabular}
\caption{Values of $\ep_{v,w}$}\label{fig.ep}
\end{figure}

The main theorem also applies to many non-compact manifolds;
see Theorem~\ref{main}.

\subsection{Further comments on  related literature}

Our main motivation to study Yamabe constants of products is the
application sketched in Subsection~\ref{subsec.surg}.

Fundamental results on Yamabe constants on products have been found in
the interesting article \cite{akutagawa.florit.petean:07}, where it is,
among other things, shown that the conformal Yamabe constant of the
product $V \times \mR^w$ is a lower bound for $\si(V \times W)$. Here
$\si$ denotes the smooth Yamabe invariant defined in
Subsection~\ref{subsec.smooth}. This article also emphasized the
importance of the question under which conditions
a function $u \in C^\infty(V\times W)$ minimizing $\cF$ is a function
of only one of the factors. In \cite{akutagawa.florit.petean:07} it is
proved that if $(V,g)$ is compact and of constant scalar
curvature $1$, then the infimum of the Yamabe functional of
$V \times \mR^w$ restricted to functions depending only on $\mR^w$ is
up to a constant the inverse of an optimal constant in a
Gagliardo-Nirenberg type estimate.

In related research, Petean  \cite{petean:09a}
derived a lower bound for the conformal
Yamabe constant of product manifolds $V\times \mR$, where $V$ is
compact of positive Ricci curvature. If additionally we require $V$ to
be Einstein, any minimizer
$u\in C^\infty(V\times \mR)$ of $\cF$ only depends on $\mR$.
As a corollary, Petean obtained
lower bounds for the smooth Yamabe invariant $\si(V \times S^1)$ in
this case.

This result of Petean contrasts nicely to  Theorem~\ref{main}. Whereas
Petean's result requires that one of the factors is $1$-dimensional,
our  Theorem~\ref{main} requires both factors to be of dimension at
least $3$.

In \cite{petean.ruiz:p10} an explicit lower bound for
$\mu(S^2\times \mR^2,\rho^2 + \xi^2)$ is obtained:
$\mu(S^2\times \mR^2, \rho^2 + \xi^2)
\geq 0.68 \cdot Y(S^4)$. 
A similar but weaker result was obtained in \cite{petean:01}.

Several recent publications study multiplicity phenomena on products
$S^v\times W$ equipped with the product metric of the standard metric on
$S^v$ with a metric of constant scalar curvature $s>0$ on
$W$. Explicit lower bounds for the number of metrics of constant
scalar curvature $1$ in the conformal class $[g_0]$ are derived, and
these bounds grow linearly in  $\sqrt{s}$. The case $v=1$ was studied
in \cite{lbb.kaas:a,lbb.kaas:b}; the general case then treated in
\cite{petean:10}. In the recent preprint \cite{henry.petean:p11},
isoparametric hypersurfaces are used in order to obtain new metrics of
constant scalar curvature in the conformal class of
products of Riemannian manifolds, e.g. the conformal class of
$(S^3\times S^3,\rho^3+\lambda \rho^3)$.

\subsection{Structure of the present article}

In Section~\ref{Section_product} we derive the main techniques and the
main result of the article. We use mixed
$L^{p,q}$-spaces in
order to obtain a lower bound of the conformal Yamabe constants in
the case that both factors have dimension at least $3$. We start with
a proof of an iterated H\"older inequality in
Subsection~\ref{iterated} which is well-adapted for the proof of our
product formula in Subsection~\ref{subsec.main}, which is the main
result of the article.

In Section~\ref{Section_applications} we discuss applications. In
Subsection~\ref{subsec.smooth} we find an estimate for the smooth
Yamabe invariant of product manifolds. Subsection~\ref{subsec.surg}
explains our original motivation for the subject, which is to find
better estimates for the constants appearing in the surgery formula in
\cite{ammann.dahl.humbert:p08a}. In Subsection~\ref{subsec.stable} we
define a stable Yamabe invariant and show that a similar surgery
formula as in the unstable situation holds true.


\section{Yamabe constants of product metrics}
\label{Section_product}

\subsection{Iterated H\"older inequality for product manifolds}
\label{iterated}

Let $(V,g)$ and $(W,h)$ be Riemannian manifolds of dimensions
$v \definedas \dim V$ and $w \definedas \dim W$. We set
\begin{equation*}
(M,G) \definedas (V \times W, g + h),
\end{equation*}
so that $m \definedas \dim M = v + w$. We do not assume that the
manifolds are complete. The first result we will need is a kind of
iterated H\"older inequality for $(M,G) \definedas
(V \times W, g + h)$.

\begin{lemma} \label{holder}
For any function $u \in C_c^{\infty}(M)$ we have
\begin{equation*}
\left( \int_M |u|^{p_m} \, dv^G \right)^{\frac{2}{p_m}}
\leq
\left( \int_V
\left( \int_W |u|^{p_w} \, dv^h \right)^{\frac{2}{p_w} }
\, dv^g \right)^{\frac{w}{m} }
\left( \int_V
\left( \int_W |u|^2 \, dv^h \right)^{\frac{p_v}{2}}
\, dv^g \right)^{\frac{v-2}{m}} .
\end{equation*}
\end{lemma}
The lemma is actually a special case of the H\"older inequality for mixed
$L^{p,q}$-spaces. See \cite{benedek.panzone:61} for further information on
such spaces.

\begin{proof}
By the H\"older inequality we have
\begin{equation*}
\int_{W} |u|^{p_m} \, dv^h
\leq
\left( \int_W |u|^{p_w} \, dv^h \right)^{\frac{w-2}{m-2}}
\left( \int_W |u|^{2} \, dv^h \right)^{\frac{v}{m-2}}.
\end{equation*}
We integrate this inequality over $(V,g)$ and use the following
H\"older inequality:
\begin{equation*}
\int_V \alpha \beta \, dv^g
\leq
\left(
\int_V |\alpha|^{\frac{m-2}{w}} \, dv^g
\right)^{\frac{w}{m-2}}
\left(
\int_V |\beta|^{\frac{m-2}{v-2}} \, dv^g
\right)^{\frac{v-2}{m-2}}
\end{equation*}
with
\begin{equation*}
\alpha \definedas
\left( \int_W |u|^{p_w} \, dv^h \right)^{\frac{w-2}{m-2}}
\; \hbox{ and } \;
\beta \definedas
\left( \int_W |u|^{2} \, dv^h \right)^{\frac{v}{m-2}}.
\end{equation*}
This proves Lemma \ref{holder}.
\end{proof}

\subsection{A lemma about integration and derivation}

Second we need a lemma concerning the interchange of derivation and
taking (a partial) $L^2$-norm.

\begin{lemma} \label{comparing_deriv}
Let $u \in C_c^\infty(M)$, $u \not\equiv 0$, and set
\begin{equation*}
\ga \definedas
\left( \int_W u^2  \, dv^h \right)^{\frac{1}{2}}.
\end{equation*}
Then
\begin{equation} \label{cdr}
\int_V |d\ga|_g^2 \, dv^g
\leq
\int_M |du|_g^2 \, dv^G.
\end{equation}
\end{lemma}

\begin{proof}
Take any vector field $X$ on $M$ tangent to $V$. One has $g$-almost
everywhere (except on the boundary of $\ga^{-1}(0)$)
\begin{equation*}
|X \ga|^2
\leq
\left(
\frac{\int_W u X u \, dv^h}
{ \left(\int_W u^2 \,dv^h \right)^\frac{1}{2} }
\right)^2
\leq
\int_W (X u)^2  \, dv^h ,
\end{equation*}
where we used the Cauchy-Schwartz inequality
\begin{equation*}
\int_W u X u \, dv^h
\leq
{\left(\int_W (X u)^2 \, dv^h\right)}^{\frac{1}{2}}
{\left(\int_W  u^2 \, dv^h\right)}^{\frac{1}{2}}.
\end{equation*}
Integrating over $V$, we deduce that
\begin{equation*}
\int_V |X \ga|^2 \, dv^g \leq \int_M |X u|^2 \, dv^G.
\end{equation*}
Since this holds for any $X$ tangent to $V$, inequality \eqref{cdr}
follows.
\end{proof}

\subsection{Conformal Yamabe constant of product metrics}
\label{subsec.main}

We now state and prove our main theorem. It will turn out that the
following modified invariant is convenient when studying products of
Riemannian manifolds with non-negative Yamabe constant. If
$\mu(M,G) \geq 0$ we set
\begin{equation*}
\nu(M,G) \definedas
\left( \frac{\mu(M,G)}{m a_m} \right)^m .
\end{equation*}

\begin{theorem} \label{main}
Let $(V,g)$ and $(W,h)$ be Riemannian manifolds of dimensions
$v,w \geq 3$. Assume that $\mu(V,g), \mu(W,h) \geq 0$ and that
\begin{equation} \label{assump}
\frac{\Scal^g + \Scal^h}{a_m}
\geq
\frac{\Scal^g}{a_v} + \frac{\Scal^h}{a_w} .
\end{equation}
Then,
\begin{equation*}
\mu(M,G) \geq
\frac{m a_m}{(v a_v)^{\frac{v}{m}} (w a_w)^{\frac{w}{m}}}
\mu(V,g)^{\frac{v}{m}} \mu(W,h)^{\frac{w}{m}},
\end{equation*}
or, equivalently,
\begin{equation*}
\nu(M,G) \geq \nu(V,g) \nu(W,h).
\end{equation*}
\end{theorem}

Note that we do not assume that the manifolds are complete.

\begin{proof}
Take any non-negative function $u \in C_c^\infty(M)$ normalized by
\begin{equation} \label{normalization}
\int_M |u|^{p_m} \, dv^G = 1 .
\end{equation}
We then have
\begin{equation*}
\frac{1}{a_m} \cF^G(u)
=
\int_M \left( |du|_G^2 +\frac{\Scal^G}{a_m}u^2 \right) \, dv^G.
\end{equation*}
Using $|du|^2_G = |du|^2_g+|du|^2_h$ and
$\Scal^G = \Scal^g + \Scal^h$ together with \eqref{assump} we obtain
\begin{equation} \label{IGu}
\frac{1}{a_m} \cF^G(u)
\geq
\int_M \left( |du|^2_g + \frac{\Scal^g}{a_v} u^2 \right) \, dv^G
+
\int_V
\int_W \left( |du|_h^2 + \frac{\Scal^h}{a_w} u^2 \right) \, dv^h
\, dv^g.
\end{equation}
We set $\ga \definedas \left( \int_W u^2 \, dv^h \right)^\frac{1}{2}$.
For the first term here, Lemma \ref{comparing_deriv} and the
definition of $\mu(V,g)$ imply that
\begin{equation} \label{IGu1}
\begin{split}
\int_M \left( |du|^2_g + \frac{\Scal^g}{a_{v}} u^2 \right) \, dv^G
&\geq
\int_V
\left( |d\ga|_g^2 + \frac{\Scal^g}{a_{v}} \ga^2 \right)
\, dv^g \\
&\geq
\frac{1}{a_v} \mu(V,g)
\left( \int_{V} \ga^{p_v} \, dv^g \right)^{\frac{2}{p_v}} \\
&=
\frac{1}{a_v} \mu(V,g)
\left(
\int_V \left( \int_W |u|^2 \, dv^h \right)^{\frac{p_v}{2}} \, dv^g
\right)^{\frac{v-2}{v} }.
\end{split}
\end{equation}
For the second term we have
\begin{equation} \label{IGu2}
\begin{split}
\int_V \int_W
\left(|du|_h^2 + \frac{\Scal^h}{a_w} u^2 \right)
\, dv^h \, dv^g
\geq
\frac{1}{a_w} \mu(W,h)
\int_V
\left( \int_W u^{p_w} \, dv^h \right)^{\frac{2}{p_w}}
\, dv^g
\end{split}
\end{equation}
by the definition of $\mu(W,h)$. Plugging \eqref{IGu1} and
\eqref{IGu2} into \eqref{IGu} we get
\begin{equation} \label{IGu3}
\begin{split}
\cF^G(u)
&\geq
\frac{a_m}{a_v} \mu(V,g)
\left(
\int_V \left( \int_W |u|^2 \, dv^h \right)^{\frac{p_v}{2}} \, dv^g
\right)^{\frac{v-2}{v} } \\
&\qquad +
\frac{a_m}{a_w} \mu(W,h)
\int_V \left( \int_W u^{p_w} \, dv^h \right)^{\frac{2}{p_w}} \, dv^g.
\end{split}
\end{equation}
Set
\begin{equation*}
r \definedas
m a_m \nu(V,g)^{\frac{1}{m}} \nu(W,h)^{\frac{1}{m}}.
\end{equation*}
For $a,b > 0$ we compute
\begin{equation*}
\begin{split}
r a^\frac{v-2}{m} b^\frac{w}{m}
&=
r
\left(
\left( \frac{\nu(V,g)^w}{\nu(W,h)^v} \right)^{\frac{1}{m^2}}
a^\frac{v-2}{m}
\right)
\left(
\left( \frac{\nu(W,h)^v}{\nu(V,g)^w} \right)^{\frac{1}{m^2}}
b^\frac{w}{m}
\right) \\
&\leq
r \left[
\frac{v}{m}
\left( \frac{\nu(V,g)^{\frac{w}{v}}}{\nu(W,h)} \right)^{\frac{1}{m}}
a^\frac{v-2}{v}
+
\frac{w}{m}
\left( \frac{\nu(W,h)^{\frac{v}{w}}}{\nu(V,g)} \right)^{\frac{1}{m}}
b
\right] \\
&=
m a_m \nu(V,g)^{\frac{1}{m}} \nu(W,h)^{\frac{1}{m}}
\frac{v}{m}
\left( \frac{\nu(V,g)^{\frac{w}{v}}}{\nu(W,h)} \right)^{\frac{1}{m}}
a^\frac{v-2}{v} \\
&\qquad +
m a_m \nu(V,g)^{\frac{1}{m}} \nu(W,h)^{\frac{1}{m}}
\frac{w}{m}
\left( \frac{\nu(W,h)^{\frac{v}{w}}}{\nu(V,g)} \right)^{\frac{1}{m}}
b \\
&=
a_m v \nu(V,g)^{\frac{1}{v}} a^\frac{v-2}{v}
+ a_m w \nu(W,h)^{\frac{1}{w}} b \\
&=
\frac{a_m}{a_v} \mu(V,g) a^\frac{v-2}{v}
+ \frac{a_m}{a_w} \mu(W,h) b, \\
\end{split}
\end{equation*}
where in the second line we used Young's inequality
\begin{equation*}
cd \leq \frac{v}{m}c^{\frac{m}{v}} + \frac{w}{m}d^{\frac{m}{w}} ,
\end{equation*}
which is valid for any $c,d \geq 0$. Using the above in \eqref{IGu3}
with
\begin{equation*}
a \definedas
\int_V
\left( \int_W |u|^2 \, dv^h \right)^{\frac{p_v}{2}}
\, dv^g
\; \hbox{ and } \;
b \definedas
\int_V
\left( \int_W |u|^{p_w} \, dv^h \right)^{\frac{2}{p_w}}
\, dv^g,
\end{equation*}
we get
\begin{equation*}
\cF^G(u)
\geq
r
\left( \int_V
\left( \int_W |u|^2 \, dv^h \right)^{\frac{p_v}{2}}
\, dv^g \right)^\frac{v-2}{m}
\left( \int_V
\left( \int_W |u|^{p_w} \, dv^h \right)^{\frac{2}{p_w}}
\, dv^g \right)^\frac{w}{m} .
\end{equation*}
Using Lemma \ref{holder} and Relation \eqref{normalization}
we deduce
\begin{equation*}
\cF^G(u)
\geq r
= m a_m \nu(V,g)^{\frac{1}{m}} \nu(W,h)^{\frac{1}{m}}
= \frac{m a_m}{(v a_v)^{\frac{v}{m}} (w a_w)^{\frac{w}{m}}}
\mu(V,g)^{\frac{v}{m}} \mu(W,h)^{\frac{w}{m}}.
\end{equation*}
Since this holds for all $u$, Theorem~\ref{main} follows.
\end{proof}

\section{Applications}
\label{Section_applications}

\subsection{The smooth Yamabe invariant of product manifolds}
\label{subsec.smooth}

Let $M$ be a compact manifold of dimension $m \geq 3$. Then its
{\it smooth Yamabe invariant} is defined as
\begin{equation*}
\sigma(M) \definedas \sup \mu(M,G),
\end{equation*}
where the supremum runs over all Riemannian metrics $G$ on $M$.
This invariant of differentiable manifolds has the property that
$\si(M) \leq \si(S^m)$ for all $M$ and $\si(M) > 0$ if and only if
$M$ admits a metric with positive scalar curvature.

{}From Theorem~\ref{main} we obtain the following corollary.
\begin{cor} \label{cor1}
Let $V,W$ be compact manifolds of dimensions $v,w \geq 3$. Assume
$\sigma(V)\geq 0$. Then
\begin{equation*}
\sigma(V \times W)
\geq
\frac{m a_m}{(v a_v)^{\frac{v}{m}} (w a_w)^{\frac{w}{m}}}
\sigma(V)
^{\frac{v}{m}}
\sigma(S^w)
^{\frac{w}{m}},
\end{equation*}
where $m = v + w$.
\end{cor}
\begin{proof}
We first consider the case $\si(V)>0$.
In \cite[Theorem 1.1]{akutagawa.florit.petean:07} it is proven that
\begin{equation*}
\lim_{t \to\infty} \mu(V \times W, g + t^2 h)
=
\mu(V \times \mR^w, g + \xi^w)
\end{equation*}
if $g$ is a metric on $V$ with positive scalar curvature and $h$ is
any metric on $W$. Since $a_v \geq a_m$ we see that
\eqref{assump} holds, so Theorem~\ref{main} together with
$\mu(\mR^w, \xi^w) = \mu(S^w, \rho^w)$ implies the corollary if $\si(V)>0$.

In the case $\si(V)=0$ there is a sequence of metrics $g_i$ on $V$
such that $\vol^{g_i}(V)=1$, $\mu(V,g_i)\leq 0$, and $\mu(V,g_i)\to 0$
as $i \to \infty$. From the solution of the Yamabe problem we can
assume that all $g_i$ have constant scalar curvature $\Scal^{g_i}
= \mu(V,g_i)$. Choose $\ep_i>0$ such that $\ep_i \to 0$ and
$\ep_i^{-w}\mu(V,g_i)\to 0$ for $i\to \infty$. For a unit volume
metric $h$ on $W$ with constant scalar curvature $\Scal^h$, the
metric $G_i \definedas \ep_i^{w} g_i + \ep_i^{-v} h$ has
$\vol^{G_i}(V\times W) = 1$ and constant scalar curvature
$\ep_i^{-w}\mu(V,g_i)+\ep_i^v \Scal^h \to 0$. It follows that
$\mu(V\times W,G_i)\to 0$ and thus $\si(V\times W)\geq 0$.
\end{proof}

\subsection{Surgery formulas}
\label{subsec.surg}

Assume that $M$ is a compact $m$-dimensional manifold and that
$i:S^k\times \overline{B^{m-k}}\to M$ is an embedding.
We define
\begin{equation*}
N \definedas
(M\setminus i(S^k\times B^{m-k})\cup_\pa (B^{k+1}\times S^{n-k-1}),
\end{equation*}
where $\cup_\pa$ means that we identify $x\in S^k\times
S^{m-k-1}=\pa(B^{k+1}\times S^{m-k-1})$
with $i(x)\in \pa i (S^k\times B^{m-k})$. After a smoothing procedure,
$N$ is again a compact manifold without boundary, and we say that
$N$ is \emph{obtained from $M$ by $m$-dimensional surgery along $i$}.

In \cite[Corollary 1.4]{ammann.dahl.humbert:p08a} we found the
following result.

\begin{theorem}\label{sigmaanatheo}
Let $N$ be obtained from $M$ via surgery of dimension
$k\in\{0,1,\ldots, \break m-3\}$. Then there is
a constant $\La_{m,k}>0$ with
\begin{equation*}
\si(N)\geq \min\{\si(M),\La_{m,k}\}.
\end{equation*}
Furthermore, for $k=0$ this statement is true for $\La_{m,0}=\infty$.
\end{theorem}

It is helpful to consider how the constant $\Lambda_{m,k}$ was
obtained in \cite{ammann.dahl.humbert:p08a} in the case $k\geq 1$. We
showed that Theorem~\ref{sigmaanatheo} holds for a constant
$\La_{m,k}$ satisfying
\begin{equation*}
\La_{m,k} \geq
\min \left\{ \La^{(1)}_{m,k},\La^{(2)}_{m,k}\right\}.
\end{equation*}
We will not recall here the definitions of $\La^{(1)}_{m,k}$ and
$\La^{(2)}_{m,k}$ in detail, as they are not needed, but we will
explain some relevant facts for $\La^{(1)}_{m,k}$ and
$\La^{(2)}_{m,k}$.

For $c\in[0,1]$, let $\mH^{k+1}_c$ be the simply connected
$(k+1)$-dimensional complete Riemannian manifold of constant sectional
curvature $-c^2$; for $c=0$ it is $\mR^{k+1}$ and for $c>0$ it is
hyperbolic space rescaled by a factor $c^{-2}$. One defines
\begin{equation*}
\La_{m,k}^{(0)}\definedas
\inf_{c \in [0,1]} \mu(\mH^{k+1}_c\times \mS^{m-k-1}).
\end{equation*}

It was shown in \cite[Corollary 1.4]{ammann.dahl.humbert:p08a} that
$\La^{(1)}_{m,k}\geq\La_{m,k}^{(0)}$ for $k\in \{1,\ldots,m-3\}$.
Furthermore $\La^{(2)}_{m,k}\geq \La^{(1)}_{m,k}$ will be shown in our
publication \cite{ammann.dahl.humbert:p11b} provided that
$k+3\leq m\leq 5$ or $k+4\leq m$. Thus Theorem~\ref{sigmaanatheo}
holds for $\La_{m,k}\definedas\La_{m,k}^{(0)}$ if $k+3 \leq m \leq 5$
or $k+4 \leq m$.


Since the manifolds $\mH^{k+1}_c\times \mS^{m-k-1}$ with $c^2 \leq 1$
satisfy Condition~\eqref{assump} and since
$\mu(\mH^{k+1}_c) = \mu(\mS^{k+1})$, we obtain the following Corollary
from Theorem~\ref{main}.

\begin{corollary}\label{cor2}
If $2 \leq k \leq m-4$, then Theorem~\ref{sigmaanatheo} holds for
\begin{equation*}
\La_{m,k}
=
\frac{m a_m}{((k+1) a_{k+1})^{\frac{k+1}{m}}
((m-k-1) a_{m-k-1})^{\frac{m-k-1}{m}}}
\sigma(S^{k+1})^{\frac{k+1}{m}}
\sigma(S^{m-k-1})^{\frac{m-k-1}{m}}\!.
\end{equation*}
\end{corollary}

It follows for example: If $M$ is an $m$-dimensional compact manifold,
obtained from $S^m$ by performing successive surgeries of dimension
$k$, $0\leq k\leq m-4$, $k\neq 1$, then $\si(M)\geq \La_m$, where
$\La_6= 54.779$, $\La_7= 74.504$, $\La_8= 92.242$, $\La_9= 109.426$, etc.

Another technique presented in \cite{ammann.dahl.humbert:p11c:sigmaexp}
will allow us to control the effect of $1$-dimensional
surgeries. In combination with \cite{petean.ruiz:p10} and
\cite{petean.ruiz:p11} one obtains $\La_{4,1}>38.9$ and $\La_{5,1}>45.1$.

\subsection{A stable Yamabe invariant}
\label{subsec.stable}

In this section we will define and discuss a ``stabilized'' Yamabe
invariant, obtained by letting the dimension go to infinity for a
given compact Riemannian manifold by multiplying with Ricci-flat
manifolds of increasing dimension. Very optimistically, such a
stabilization could be related to the linear eigenvalue problem
obtained by formally letting the dimension tend to infinity in the
Yamabe problem. The stable invariant can also be viewed as a
quantitative refinement of the property that a given manifold admit
stably positive scalar curvature.

For a compact manifold $M$ with $\si(M) \geq 0$ we define
\begin{equation*}
\Si(M) \definedas
\left( \frac{\si(M)}{m a_m} \right)^m .
\end{equation*}
Then
\begin{equation*}
\Si(M) = \sup \nu(M,G),
\end{equation*}
where the supremum runs over all Riemannian metrics $G$ on $M$.
The conclusion of Corollary \ref{cor1} can be formulated as
\begin{equation} \label{Sigma_product}
\Si (V \times W)
\geq
\Si(V) \Si(S^w).
\end{equation}

Let $(B,\beta)$ be a compact Ricci-flat manifold of dimension $b$. We
could for example choose $B$ to be the $1$-dimensional circle $S^1$,
or an $8$-dimensional Bott manifold equipped with a metric with
holonomy ${\rm Spin}(7)$. From \eqref{Sigma_product} we then get
\begin{equation} \label{Si_est2}
\frac{\Si(S^{v + bi})}{\Si(S^{bi})}
\geq
\frac{\Si (V \times B^i)}{\Si(S^{bi})}
\geq
\Si(V),
\end{equation}
where the upper bound comes from $\Si (V \times B^i) \leq
\Si(S^{v + bi})$. We define the stable Yamabe invariant of $V$ as the
limit superior of the middle term,
\begin{equation*}
\overline{\Si} ( V )
\definedas
\limsup_{i \to \infty}
\frac{\Si (V \times B^i)}{\Si(S^{bi})}.
\end{equation*}

To see that the stable Yamabe invariant is finite we need to study the
upper bound in \eqref{Si_est2} and the function
$v \mapsto \Si(S^v)$. We have
\begin{equation*}
\si(S^v) = v(v-1) \om_v^{2/v} , \quad
\om_v =
\frac{2 \pi^{\frac{v+1}{2}}}{\Gamma \left( \frac{v+1}{2} \right)},
\end{equation*}
where $\om_v$ is the volume of $\mS^v$, so
\begin{equation*}
\Si(S^v)
=
4 \pi \left( \frac{\pi(v-2)}{4} \right)^v
\frac{1}{ \Gamma \left( \frac{v+1}{2} \right)^2 } .
\end{equation*}
Stirling's formula tells us that
\begin{equation*}
\Gamma(z)
=
\sqrt{\frac{2\pi}{z}} \left( \frac{z}{e} \right)^z
\left( 1 + O \left(\frac{1}{z}\right) \right)
\end{equation*}
and therefore
\begin{equation*}
\begin{split}
\Si(S^v)
&=
4 \pi \left( \frac{\pi(v-2)}{4} \right)^v
\frac{v+1}{4\pi}
\left( \frac{2e}{v+1} \right)^{v+1}
\left( 1 + O \left(\frac{1}{v}\right) \right) \\
&=
2e \left( \frac{\pi e}{2} \right)^v
\frac{(1 - 2/v)^v}{(1 + 1/v)^v}
\left( 1 + O \left(\frac{1}{v}\right) \right) \\
&=
2e^{-2} \left( \frac{\pi e}{2} \right)^v
\left( 1 + O \left(\frac{1}{v}\right) \right).
\end{split}
\end{equation*}
We see that
\begin{equation*}
\lim_{i \to \infty} \frac{\Si(S^{v + bi})}{\Si(S^{bi})}
=
\lim_{i \to \infty}
\left( \frac{\pi e}{2} \right)^v
\left( 1 + O \left( \frac{1}{bi} \right) \right)
=
\left( \frac{\pi e}{2} \right)^v,
\end{equation*}
so from \eqref{Si_est2} we get the following bound on the stable
Yamabe invariant:
\begin{equation*}
\left( \frac{\pi e }{2} \right)^v
\geq
\overline{\Si} ( V )
\geq
\Si(V).
\end{equation*}
We conclude that the stable invariant is a non-trivial invariant.

The stable Yamabe invariant is not strictly speaking a stable
invariant in the sense that it gives the same value for $V$ and
$V \times B^i$. These values are however related by a simple identity,
as we will see next. Taking the limit superior as $j \to \infty$ in
\begin{equation*}
\frac{\Si (V \times B^i \times B^j)}{\Si(S^{bj})}
=
\frac{\Si (V \times B^{i+j})}{\Si(S^{bi+bj})}
\frac{\Si(S^{bi+bj})}{\Si(S^{bj})},
\end{equation*}
we conclude that
\begin{equation*}
\overline{\Si} (V \times B^i)
=
\overline{\Si} (V) \left( \frac{\pi e }{2} \right)^{bi}
\end{equation*}
and further
\begin{equation} \label{prod.prep}
\overline{\Si} (V)
\geq
\Si (V \times B^i)
\left( \frac{\pi e }{2} \right)^{-bi}
\end{equation}
for all $i \geq 0$.

The next simple proposition tells us that the positivity of $\Si(V)$ is
equivalent to $V$ having stable metrics of positive scalar curvature.

\begin{proposition}
Let $V$ be a compact manifold. The following three statements are
equivalent.
\begin{enumerate}[(a)]
\item $\overline{\Si} (V) > 0$.
\item There is $i_0 > 0$ such that $V \times B^{i_0}$ admits a positive
  scalar curvature metric.
\item There is a $i_0 > 0$ such that $V \times B^i$ admits a positive
  scalar curvature metric for all $i \geq i_0$.
\end{enumerate}
\end{proposition}


\begin{proof}
The implications $(a) \Rightarrow (b)$ and $(b) \Leftrightarrow (c)$ are
easy to show. The implication $(b) \Rightarrow (a)$ is a consequence of
\eqref{prod.prep}.
\end{proof}

We also obtain a stable version of Theorem~\ref{sigmaanatheo} for
surgeries of codimension at least $4$. A similar result holds for
surgeries of codimension $3$, but with a less explicit constant.

\begin{theorem}
Assume that $N$ is obtained from the compact $m$-dimensional manifold
$M$ by surgery of dimension $k$, where $0 \leq k \leq m-4$. Then
\begin{equation*}
\overline{\Si}(N)
\geq
\min\left\{\overline{\Si}(M), \overline{\Si}(S^m),
\left(\frac{\pi e}2\right)^{k+1}\Si (S^{m-k-1})
\right\}.
\end{equation*}
\end{theorem}

\begin{proof}
The manifold $N$ after surgery is obtained by a connected sum of $M$
and $S^m$ along embeddings of a $k$-dimensional sphere with trivial
normal bundle. Similarly $N \times B^i$ is obtained by a connected sum
of $M \times B^i$ and $S^m \times B^i$ by a connected sum along
embeddings of $S^k \times B^i$ with trivial normal bundle. Thus
\cite[Theorem~1.3]{ammann.dahl.humbert:p08a} together with
Corollary~\ref{cor2} tells us that
\begin{equation*}
\begin{split}
\Si(N \times B^i)
&\geq
\min \left\{ \Si(M \times B^i), \Si(S^m \times B^i),
\left(\frac{\La_{m+bi,k+bi}}{(m+bi)a_{m+bi}}\right)^{m+bi} \right\}\\
&\geq
\min\{\Si(M \times B^i), \Si(S^m \times B^i),
\Si(S^{k+bi+1})\Si(S^{m-k-1})\}
\end{split}
\end{equation*}
and this yields the statement of the theorem.
\end{proof}

For the smooth Yamabe invariant the value of the sphere is a universal
upper bound. One can ask if the same holds for the stable invariant:
is $\overline{\Si}(M) \leq \overline{\Si}(S^m)$ for all $M$?

\subsection*{Acknowledgements}

We want to thank the organizers of the Conference
``Contributions in Differential Geometry -- a round table on the occasion
of the 65th birthday of Lionel B\'erard Bergery, 2010'', in particular
thanks to A. Besse, A. Altomani, T. Krantz and M.-A. Lawn. During that
conference we found central ingredients for the present article.
We thank P. Piccione for pointing out an error
in a preprint version of this article. Further, we would like to thank
the anonymous referee for many helpful comments.

\providecommand{\bysame}{\leavevmode\hbox to3em{\hrulefill}\thinspace}
\providecommand{\MR}{\relax\ifhmode\unskip\space\fi MR }
\providecommand{\MRhref}[2]{%
  \href{http://www.ams.org/mathscinet-getitem?mr=#1}{#2}
}
\providecommand{\href}[2]{#2}


\end{document}